\documentclass[reqno,12pt,letterpaper]{amsart}
\usepackage[proof]{sdmacros}

\def\pv{\mathrm{p.v.}}
\def\reg{\mathrm{reg}}
\def\gr{\mathrm{gr}}

\title[Notes on Cohen's paper]%
{Fractal uncertainty in higher dimensions:\\
notes on Cohen's paper}
\author{Semyon Dyatlov}
\email{dyatlov@math.mit.edu}
\address{Department of Mathematics, Massachusetts Institute of Technology, Cambridge, MA 02139}

\begin{document}

\begin{abstract}
In~\cite{Cohen-FUP-v1}, Cohen proved a higher dimensional fractal uncertainty principle for line porous sets.
The purpose of this expository note is to provide a different point of view
on some parts of Cohen's proof, particularly suited to readers familiar with the theory
of distributions. It is meant to be complementary to Cohen's paper.
\end{abstract}

\maketitle

\addtocounter{section}{1}
\addcontentsline{toc}{section}{1. Introduction}

In the recent paper~\cite[Theorem~1.2]{Cohen-FUP-v1}, Cohen proved the following remarkable
\begin{theo}
  \label{t:main}
Assume that $\mathbf X,\mathbf Y\subset \mathbb R^d$ are two sets such that
for some $h\in (0,{1\over 100})$ and $\nu\in (0,{1\over 10})$
\begin{itemize}
\item $\mathbf X\subset B(0,1)$ is $\nu$-porous on balls on scales $h$ to~1, and
\item $\mathbf Y\subset B(0,h^{-1})$ is $\nu$-porous on lines on scales 1 to~$h^{-1}$.
\end{itemize}
Then we have the following estimate for all $f\in L^2(\mathbb R^d)$
\begin{equation}
  \label{e:main}
\supp\widehat f\subset \mathbf Y\quad\Longrightarrow\quad
\|\indic_{\mathbf X}f\|_{L^2(\mathbb R^d)}\leq Ch^\beta \|f\|_{L^2(\mathbb R^d)}
\end{equation}
where the constants $C,\beta>0$ depend only on~$\nu,d$ but not on~$h$.
\end{theo}
Here $B(\mathbf x,r)$ denotes the $r$-ball centered at $\mathbf x$ and the notion of porosity is given by
\begin{defi}
\label{d:porous}
Assume that $0<\rho_0\leq\rho_1$, $\nu\in(0,{1\over 10})$, and $\mathbf X\subset\mathbb R^d$.

\begin{itemize}
\item We say that $\mathbf X$ is \emph{$\nu$-porous on balls on scales $\rho_0$ to~$\rho_1$}
if for each ball $B$ of diameter $R\in [\rho_0,\rho_1]$ there exists
$\mathbf x\in B$ such that
$\mathbf X\cap B(\mathbf x,\nu R)=\emptyset$.
\item We say that $\mathbf X$ is \emph{$\nu$-porous on lines on scales $\rho_0$ to~$\rho_1$}
if for each line segment~$\tau$ of length $R\in [\rho_0,\rho_1]$ there exists
$\mathbf x\in \tau$ such that $\mathbf X\cap B(\mathbf x,\nu R)=\emptyset$.
\end{itemize}
\end{defi}
Note that porosity on lines implies porosity on balls but the converse is not true when $d\geq 2$:
any line in $\mathbb R^d$ is $\frac1{10}$-porous on balls on all scales
but not porous on lines.

Our statement of Theorem~\ref{t:main} (as well as the statements of various results below) differs from the one in~\cite{Cohen-FUP-v1} in minor details, in particular we use the following normalization of the Fourier transform following~\cite{Hormander1}:
$$
\widehat f(\xi)=\int_{\mathbb R^d}e^{-i\mathbf x\cdot\xi}f(\mathbf x)\,d\mathbf x.
$$
This expository note gives a different point of view on some parts of Cohen's proof of Theorem~\ref{t:main}.
We will freely use the theory of distributions presented for example in~\cite{Hormander1},
which is not necessary (as it is not used in~\cite{Cohen-FUP-v1}) but is convenient
in particular because of the distributional characterization of plurisubharmonicity
(Proposition~\ref{l:psh-crit} below).

This note is meant to be complementary to Cohen's paper and in particular
contains no discussion of applications of fractal uncertainty principle or previous results;
the latter can be found for instance in~\cite{FUP-ICMP}. We also only present the proof of
a part of the argument, namely Theorems~\ref{t:existence-weight}
and~\ref{t:psh-constr} below
(it is possible that a future version of this note will contain other parts as well).

We restrict to the case of dimension $d=2$. This is only for the purpose
of notational convenience, shortening various formulas along the way.
For the case of dimension $d=1$, Theorem~\ref{t:main} was proved by Bourgain--Dyatlov~\cite{fullgap}
and Cohen's result~\cite{Cohen-FUP-v1} is a far-reaching generalization of this
result to all dimensions.

\subsection{Existence of special compactly supported functions}

Theorem~\ref{t:main} is deduced in~\cite{Cohen-FUP-v1} from the result of Han--Schlag~\cite{Han-Schlag}
and a novel construction of compactly supported functions with Fourier transforms decaying on line porous sets~\cite[Proposition~1.7]{Cohen-FUP-v1}. We state a slightly weaker version of the latter construction as Theorem~\ref{t:cutoff-total} below.
(Comparing to~\cite{fullgap}, the paper~\cite{Han-Schlag} generalizes~\cite[\S\S3.2--3.4]{fullgap}
and Theorem~\ref{t:cutoff-total} generalizes~\cite[Lemma~3.1]{fullgap}.)
We denote $\langle \mathbf x\rangle:=\sqrt{1+|\mathbf x|^2}$.
\begin{theo}
\label{t:cutoff-total}
Assume that $\mathbf Y\subset B(0,h^{-1})$ is $\nu$-porous on lines on scales 1 to~$h^{-1}$.
Then for each $c_{\supp}>0$ there exist $c_2,c_3>0$ and $\alpha\in (0,1)$ depending only
on $\nu,c_{\supp}$ and there exists a function $g\in L^2(\mathbb R^2)$ satisfying the following conditions:
\begin{align}
\label{e:cutoff-total-1}
\supp \widehat g &\subset B(0,c_{\supp}),\\
\label{e:cutoff-total-2}
\|g\|_{L^2(B(0,1))}&\geq c_2,\\
\label{e:cutoff-total-3}
|g(\mathbf x)|&\leq 1 \quad\text{for all }\mathbf x\in\mathbb R^2,\\
\label{e:cutoff-total-4}
|g(\mathbf x)|&\leq \exp\bigg(-c_3{|\mathbf x|\over (\log|\mathbf x|)^\alpha}\bigg)\quad\text{
for all }\mathbf x\in \mathbf Y\setminus B(0,10). 
\end{align}
\end{theo}
We note that for $\alpha>1$, one can construct a function satisfying conditions~\eqref{e:cutoff-total-1}--\eqref{e:cutoff-total-4} with $\mathbf Y=\mathbb R^2$. However, this is impossible
for $\alpha<1$, in fact if $g$ satisfies the condition~\eqref{e:cutoff-total-4}
on the entire $\mathbb R^2$ then $\widehat g$ satisfies a unique continuation principle and thus cannot
be compactly supported. The argument of~\cite{Han-Schlag} uses this unique continuation principle
and induction on scales to show that Theorem~\ref{t:cutoff-total} implies Theorem~\ref{t:main}.

In~\cite{fullgap}, the analog of Theorem~\ref{t:cutoff-total} in the case of dimension $d=1$
was proved by constructing a weight $\omega\leq 0$ which behaves like $-|x|/(\log|x|)^\alpha$
for $x\in \mathbf Y\setminus B(0,10)$ but also satisfies $\int_{\mathbb R}\langle x\rangle^{-2}\omega(x)\,dx>-\infty$; this was possible due to the porosity of~$\mathbf Y$.
Then one used the Beurling--Malliavin Theorem to construct a function $g$ satisfying~\eqref{e:cutoff-total-1}--\eqref{e:cutoff-total-3} and $|g|\leq \exp(c_3\omega)$, which implies~\eqref{e:cutoff-total-4}.

In~\cite{Cohen-FUP-v1} Cohen also constructs a weight $\omega$ with the right kind of behavior
on~$\mathbf Y$ (stated as Theorem~\ref{t:existence-weight} below) and satisfying the right conditions so that one can construct a function $g$ satisfying~\eqref{e:cutoff-total-1}--\eqref{e:cutoff-total-3}
and $|g|\leq \exp(c_3\omega)$ (stated as Theorem~\ref{t:cutoff-functions} below). We now describe
these `right conditions', or rather their slight modification making the split into dyadic
pieces explicit. For that we need a few definitions:
\begin{itemize}
\item
Define the \emph{Kohn--Nirenberg symbolic norm} of order~1 with 3 derivatives on the space
of smooth compactly supported functions~$\CIc(\mathbb R^2)$ as follows:
\begin{equation}
  \label{e:KN-def}
\|\psi\|_{S^1_3}:=\max_{|\gamma|\leq 3}\sup_{\mathbf x\in\mathbb R^2}|\langle \mathbf x\rangle^{|\gamma|-1}\partial_{\mathbf x}^\gamma\psi(\mathbf x)|.
\end{equation}
\item 
For a function $\psi\in \CIc(\mathbb R^2)$, define its \emph{X-ray transform} as the
function $T\psi\in \CIc(\mathbb R_s\times \mathbb S^1_\theta)$, where $\mathbb S^1=\mathbb R/2\pi\mathbb Z$, given by the formula
\begin{equation}
  \label{e:X-ray-def}
T\psi(s,\theta):=\int_{\mathbb R}\psi(t\cos\theta-s\sin\theta,t\sin\theta+s\cos\theta)\,dt.
\end{equation}
Note that $T\psi(s,\theta)$ is the integral of $\psi$ over a line in the direction
given by the angle~$\theta$ which is distance $|s|$ from the origin~-- see Figure~\ref{f:X-ray}.
In particular, $T\psi(0,\theta)$ is the integral of~$\psi$ over a line passing through the origin.
\end{itemize}
\begin{figure}
\includegraphics{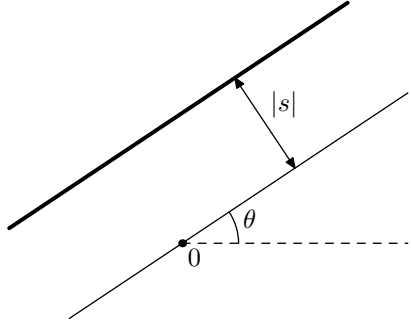}
\caption{An illustration of the X-ray transform: $T\psi(s,\theta)$ is the integral of $\psi$ over the thick line.}
\label{f:X-ray}
\end{figure}
The weight $\omega$ will have the form
\begin{equation}
  \label{e:omega-form}
\omega=\sum_{k\geq 1}\omega_k\quad\text{where }
\omega_k\in\CIc(\mathbb R^2;(-\infty,0]),\quad
k\in\mathbb N.
\end{equation}
Here the summands $\omega_k$ satisfy the following conditions
for some constants $C_{\reg},C_{\gr}\geq 0$:
\begin{align}
\label{e:omega-k-ass-1}
\omega_k&=0\quad\text{for $k$ large enough},\\
\label{e:omega-k-ass-2}
\supp\omega_k&\subset \{2^{k-1}\leq |\mathbf x|\leq 2^{k+1}\},\\
\label{e:omega-k-ass-3}
\|\omega_k\|_{S^1_3}&\leq C_{\reg},\\
\label{e:omega-k-ass-4}
\sum_{k=1}^\infty q_k&\leq C_{\gr}\quad\text{where }
q_k:=-\inf_{\theta\in\mathbb S^1}T(|\mathbf x|^{-2}\omega_k)(0,\theta).
\end{align}
We give a few comments on the above conditions:
\begin{itemize}
\item \eqref{e:omega-k-ass-1} ensures that $\omega\in\CIc(\mathbb R^2)$. It is made for technical convenience. Note that the constants in the results below do not depend on how many of
the functions~$\omega_k$ are nonzero.
\item \eqref{e:omega-k-ass-2} requires that each $\omega_k$ is supported in a dyadic annulus.
\item \eqref{e:omega-k-ass-3} is a Kohn--Nirenberg type regularity condition on each $\omega_k$,
in particular ensuring that the total weight $\omega$ is bounded in $S^1_3$. Given~\eqref{e:omega-k-ass-2}, we can rewrite it in the following form:
\begin{equation}
\label{e:omega-k-ass-3x}
\sup|\partial^\gamma_{\mathbf x}\omega_k|\leq \widetilde C_{\reg} 2^{(1-|\gamma|)k}\quad\text{for }|\gamma|\leq 3.
\end{equation}
\item \eqref{e:omega-k-ass-4} is a condition on the growth of $\omega_k$ along lines passing through
the origin. It implies in particular that $T(|\mathbf x|^{-2}\omega)(0,\theta)\geq -C_{\gr}$
for all $\theta$ but it is stronger due
to the infimum in the definition of~$q_k$. Note also that each $q_k$ is bounded in terms of the
regularity constant $C_{\reg}$ alone, more precisely we have
\begin{equation}
  \label{e:q-k-apriori}
0\leq q_k\leq 10C_{\reg}.
\end{equation}
The role of the growth constant $C_{\gr}$ is to control the sum of $q_k$ over the dyadic scales. 
\end{itemize}
We now state two theorems which together give Theorem~\ref{t:cutoff-total}. The first one (similar to the result of~\cite[\S6.1]{Cohen-FUP-v1} and proved in~\S\ref{s:weight-constr} below) is the existence of a weight satisfying the above conditions and
also sufficiently negative on a line porous set:
\begin{theo}
  \label{t:existence-weight}
Assume that $\mathbf Y\subset B(0,h^{-1})$ is $\nu$-porous on lines on scales $1$ to $h^{-1}$.
Then there exists a weight $\omega$ of the form~\eqref{e:omega-form} such that:
\begin{itemize}
\item
the summands $\omega_k$ satisfy the conditions~\eqref{e:omega-k-ass-1}--\eqref{e:omega-k-ass-4}
for some constants $C_{\reg},C_{\gr}$ depending only on the porosity constant~$\nu$, and
\item there exists $\alpha\in (0,1)$ depending only on $\nu$ such that
\begin{equation}
  \label{e:omega-large-Y}
\omega(\mathbf x)\leq -{|\mathbf x|\over (\log|\mathbf x|)^{\alpha}}\quad\text{for all }
\mathbf x\in \mathbf Y\setminus B(0,10).
\end{equation}
\end{itemize}
\end{theo}
The second one, which is perhaps the main novelty of Cohen's paper and is
a result of independent interest, is the following higher dimensional version
of the Beurling--Malliavin Theorem, see~\cite[Theorem~1.4]{Cohen-FUP-v1}:
\begin{theo}
\label{t:cutoff-functions}
Assume that $\omega$ is given by~\eqref{e:omega-form} where $\omega_k$ satisfy the
conditions~\eqref{e:omega-k-ass-1}--\eqref{e:omega-k-ass-4} for some constants $C_{\reg},C_{\gr}$.
Fix also $c_{\supp}>0$. Then there exist $c_2,c_3>0$ depending only on $C_{\reg},C_{\gr},c_{\supp}$
and there exists $g\in L^2(\mathbb R^2)$ which satisfies the conditions~\eqref{e:cutoff-total-1}--\eqref{e:cutoff-total-2} and
\begin{equation}
  \label{e:cutoff-total-omega}
|g(\mathbf x)|\leq \exp(c_3\omega(\mathbf x))\quad\text{for all }\mathbf x\in\mathbb R^2. 
\end{equation}
\end{theo}
We note that the analog of Theorem~\ref{t:cutoff-functions} in dimension $d=1$
is weaker than the Beurling--Malliavin Theorem due to the regularity assumption~\eqref{e:omega-k-ass-3}, but it suffices for the purposes of proving a Fractal Uncertainty Principle
as shown by Theorem~\ref{t:existence-weight}.

\subsection{Constructing plurisubharmonic functions on~$\mathbb C^2$}

We now discuss the proof of Theorem~\ref{t:cutoff-functions}.
By the Paley--Wiener Theorem, the compact Fourier support condition~\eqref{e:cutoff-total-1}
is essentially equivalent to the function $g$ having a holomorphic extension
$G:\mathbb C^2\to\mathbb C$ which satisfies the bound for some $C\geq 0$
\begin{equation}
  \label{e:PW-bound}
\log|G(\mathbf x+i\mathbf y)|\leq c_{\supp}|\mathbf y|+C\quad\text{for all }
\mathbf x,\mathbf y\in\mathbb R^2.
\end{equation}
In terms of the function $G$, the condition~\eqref{e:cutoff-total-omega} becomes
\begin{equation}
  \label{e:cutoff-total-omega2}
\log|G(\mathbf x)|\leq c_3\omega(\mathbf x)\quad\text{for all }\mathbf x\in\mathbb R^2.
\end{equation}
Note that~\eqref{e:cutoff-total-omega2} is stronger than~\eqref{e:PW-bound}
on~$\mathbb R^2$ (since $\omega\leq 0$) but~\eqref{e:PW-bound} needs to hold on the entire~$\mathbb C^2$.

Since $G$ is holomorphic, the function $\log|G|$ is plurisubharmonic on~$\mathbb C^2$.
Thus the existence of a nonzero function $g$ satisfying~\eqref{e:cutoff-total-1} and~\eqref{e:cutoff-total-omega} implies that there exists a plurisubharmonic function $\widetilde u:\mathbb C^2\to[-\infty,\infty)$ satisfying
\begin{align}
\label{e:u-cond-1}
\widetilde u(\mathbf x+i\mathbf y)&\leq c_3^{-1}c_{\supp}|\mathbf y|\quad\text{for all }\mathbf x,\mathbf y\in\mathbb R^2,\\
\label{e:u-cond-2}
\widetilde u(\mathbf x)&\leq \omega(\mathbf x)\quad\text{for all }\mathbf x\in\mathbb R^2.
\end{align}
Namely, one can take $\widetilde u:=c_3^{-1}(\log|G|-C)$.

Following an unpublished note by Bourgain, Cohen uses H\"ormander's $\bar\partial$-theorem
to show the converse statement: starting from a plurisubharmonic function $\widetilde u$
satisfying the bounds~\eqref{e:u-cond-1}--\eqref{e:u-cond-2} one can construct
a holomorphic function $G$ satisfying~\eqref{e:PW-bound}--\eqref{e:cutoff-total-omega2}
and obtain Theorem~\ref{t:cutoff-functions}. This is described in~\cite[\S5]{Cohen-FUP-v1}
and we do not give the details here (hopefully this part of the argument will be
in the next version of this note).

Thus Theorem~\ref{t:cutoff-functions} is reduced to constructing a plurisubharmonic function~$\widetilde u$
satisfying~\eqref{e:u-cond-1}--\eqref{e:u-cond-2}.
To do the latter, Cohen introduces the extension operator
\begin{equation}
  \label{e:E-prop}
E:\CIc(\mathbb R^2)\to C^0(\mathbb C^2)
\end{equation}
defined as follows: for $\omega\in\CIc(\mathbb R^2)$, 
\begin{equation}
  \label{e:E-def}
E\omega(\mathbf x+i\mathbf y):={1\over \pi}\int_{\mathbb R} {\omega(\mathbf x+t\mathbf y)\over 1+t^2}\,dt,\quad
\mathbf x,\mathbf y\in\mathbb R^2.
\end{equation}
Here continuity of $E\omega$ follows from the Dominated Convergence Theorem; in fact, differentiating
under the integral sign, we also have $E\omega\in C^\infty(\mathbb C^2\setminus \mathbb R^2)$. Moreover, $E\omega$ is an extension of $\omega$, namely
\begin{equation}
  \label{e:E-ext}
E\omega(\mathbf x)=\omega(\mathbf x)\quad\text{for all }\mathbf x\in\mathbb R^2.
\end{equation}
Now, the needed function $\widetilde u$ will have the form
\begin{equation}
\label{e:u-intro-form}
\widetilde u(\mathbf x+i\mathbf y):=E\widetilde\omega(\mathbf x+i\mathbf y)+\sigma|\mathbf y|,\quad
\mathbf x,\mathbf y\in\mathbb R^2
\end{equation}
where $\widetilde\omega\in\CIc(\mathbb R^2)$ is an auxiliary weight which satisfies~$\widetilde\omega\leq\omega\leq 0$
and $\sigma\geq 0$ is a constant.
We have $E\widetilde\omega\leq 0$, so
$\widetilde u$ satisfies~\eqref{e:u-cond-1} if we put $c_3:=c_{\supp}/\sigma$.
It also satisfies~\eqref{e:u-cond-2} as follows from~\eqref{e:E-ext}.

Thus we need to show that when the weight $\omega$ satisfies the conditions of Theorem~\ref{t:cutoff-functions} and $\sigma$ is large enough, the function~$\widetilde u$ is plurisubharmonic.
This is done in the following
\begin{theo}
\label{t:psh-constr}
Assume that $\omega$ is given by~\eqref{e:omega-form} where $\omega_k$ satisfy the
conditions~\eqref{e:omega-k-ass-1}--\eqref{e:omega-k-ass-4} for some constants $C_{\reg},C_{\gr}$.
Then there exists $\widetilde\omega\in \CIc(\mathbb R^2)$ such that
$\widetilde\omega\leq \omega$ everywhere and for $\sigma\geq 0$ large enough
depending only on $C_{\reg},C_{\gr}$ the function $\widetilde u$ defined by~\eqref{e:u-intro-form}
is plurisubharmonic on~$\mathbb C^2$.
\end{theo}
The proof of Theorem~\ref{t:psh-constr} is presented in this note in~\S\S\ref{s:psh-weight}--\ref{s:psh-constr}.
It largely follows Cohen's original proof in~\cite{Cohen-FUP-v1}, with a few changes:
\begin{itemize}
\item we specialize to the case of dimension 2;
\item we use the distributional criterion for plurisubharmonicity, which
eliminates Cohen's condition on Hilbert transforms on lines (see Remark~\ref{r:Hilbert} below);
\item we explicitly write the weight $\omega$ as a dyadic sum~\eqref{e:omega-form},
which makes the conditions of Theorem~\ref{t:psh-constr} more complicated
but shortens the proof a bit, and is acceptable for applications to FUP since
the weight in Theorem~\ref{t:existence-weight} is constructed as a dyadic sum as well.
\end{itemize}

\section{Preliminaries}

\subsection{X-ray transform}

We first state several basic properties of the X-ray transform defined in~\eqref{e:X-ray-def}.
Throughout this section we assume that $\psi\in\CIc(\mathbb R^2)$.
First of all, we have the discrete symmetry
\begin{equation}
  \label{e:XR-prop-discrete}
T\psi(-s,\theta)=T\psi(s,\theta+\pi).
\end{equation}
We also have equivariance under rotations: if $\mathcal R_\gamma:\mathbb R^2\to\mathbb R^2$
is the linear map given by
\begin{equation}
  \label{e:R-gamma-def}
\mathcal R_\gamma(x_1,x_2)=(x_1\cos\gamma-x_2\sin\gamma,x_1\sin\gamma+x_2\cos\gamma)
\end{equation}
then we have
\begin{equation}
    \label{e:XR-prop-rotate}
T(\psi\circ \mathcal R_\gamma)(s,\theta)=T\psi(s,\theta+\gamma).
\end{equation}
We next prove two identities featuring the X-ray transform and the Laplace
operator $\Delta=\partial_{x_1}^2+\partial_{x_2}^2$ on~$\mathbb R^2$.
First of all, for $\gamma\in\mathbb S^1$ define the constant vector field on~$\mathbb R^2$
$$
V_\gamma:=(\cos\gamma)\partial_{x_1}+(\sin\gamma)\partial_{x_2}.
$$
We treat $V_\gamma$ as a first order differential operator.

Our first identity is used in~\S\ref{s:psh-crit-proof}.
\begin{lemm}
  \label{l:XR-prop-1}
We have for all $\gamma\in\mathbb S^1$
\begin{equation}
  \label{e:XR-prop-1}
T(V_\gamma^2\psi)(s,\theta)=\sin^2(\theta-\gamma)T(\Delta\psi)(s,\theta).
\end{equation}
\end{lemm}
\begin{proof}
By the rotational equivariance~\eqref{e:XR-prop-rotate}
and the identity $(V_\gamma\psi)\circ R_\theta=V_{\gamma-\theta}(\psi\circ R_\theta)$ it suffices to consider the case $\theta=0$. Integrating by parts in~$x_1$
and using the fact that $\psi$ is compactly supported, we get
\begin{align}
  \label{e:XR-prop-int-1}
T(V_\gamma^2\psi)(s,0)&=\int_{\mathbb R}V_\gamma^2\psi(x_1,s)\,dx_1=\sin^2\gamma\int_{\mathbb R}\partial_{x_2}^2\psi(x_1,s)\,dx_1,\\
T(\Delta\psi)(s,0)&=\int_{\mathbb R}\Delta\psi(x_1,s)\,dx_1=\int_{\mathbb R}\partial_{x_2}^2\psi(x_1,s)\,dx_1
\end{align}
which together give~\eqref{e:XR-prop-1}.
\end{proof}
The next identity is used in~\S\ref{s:weight-cr}.
\begin{lemm}
\label{l:XR-prop-2}
Assume that $\psi\in\CIc(\mathbb R^2\setminus\{0\})$. Then we have for all $\theta\in\mathbb S^1$
\begin{equation}
    \label{e:XR-prop-2}
T(\Delta\psi)(0,\theta)=(\partial_\theta^2+1)T(|\mathbf x|^{-2}\psi)(0,\theta).
\end{equation}
\end{lemm}
\begin{proof}
By rotational equivariance~\eqref{e:XR-prop-rotate} it suffices to show~\eqref{e:XR-prop-2} when $\theta=0$.
For all $\theta\in\mathbb S^1$ we have
$$
T(|\mathbf x|^{-2}\psi)(0,\theta)=\int_{\mathbb R}t^{-2}\psi(t\cos\theta,t\sin\theta)\,dt.
$$
Differentiating this twice at $\theta=0$ we get
$$
\begin{aligned}
    (\partial_\theta^2+1)T(|\mathbf x|^{-2}\psi)(0,0)&=
    \int_{\mathbb R}\partial_{x_2}^2\psi(t,0)-t^{-1}\partial_{x_1}\psi(t,0)+t^{-2}\psi(t,0)\,dt\\
    &=\int_{\mathbb R}(\Delta\psi)(t,0)\,dt=T(\Delta\psi)(0,0)
\end{aligned}
$$
where in the last line we used that
$$
\int_{\mathbb R}\partial_t\big(\partial_{x_1}\psi(t,0)+t^{-1}\psi(t,0)\big)\,dt=0.
$$
\end{proof}

\subsection{X-ray transform and Hilbert transform}

We next explain how the X-ray transform is related to the Hilbert transform.
This is only used in Remark~\ref{r:Hilbert} rather than in the main proof.

The Hilbert transform is defined as the following convolution:
$$
\mathcal H:\CIc(\mathbb R)\to C^\infty(\mathbb R)\cap L^\infty(\mathbb R),\quad
\mathcal Hf={1\over \pi}\Big(\pv {1\over x}\Big)*f
$$
where the principal value distribution $\pv{1\over x}\in\mathscr S'(\mathbb R)$ is the (distributional)
derivative of $\log|x|\in L^1_{\loc}(\mathbb R)$. Denote by $\mathcal Hf'$ the derivative of $\mathcal Hf$.

We show that if $\psi\in\CIc(\mathbb R^2)$ and $\ell$ is a line in $\mathbb R^2$ passing through a point $\mathbf x_0$, then
the Hilbert transform of the restriction $\psi|_{\ell}$ at $\mathbf x_0$ can be expressed as a weighted integral of the X-ray transform
$T(\Delta \psi)$ over the lines passing through $\mathbf x_0$. For simplicity we restrict to the case
when $\mathbf x_0=0$ and the line $\ell$ is horizontal:
\begin{lemm}
  \label{l:XR-Hilbert}
Let $\psi\in\CIc(\mathbb R^2)$ and $\varphi\in\CIc(\mathbb R)$ be defined by $\varphi(x_1):=\psi(x_1,0)$. Then
\begin{equation}
  \label{e:XR-Hilbert}
\mathcal H\varphi'(0)=-{1\over 4\pi}\int_{\mathbb S^1}|\sin\theta|\,T(\Delta\psi)(0,\theta)\,d\theta.
\end{equation}
\end{lemm}
\begin{rema}
In case $\supp\psi\subset \mathbb R^2\setminus \{0\}$, \eqref{e:XR-Hilbert} follows from~\eqref{e:XR-prop-2},
integration by parts in~$\theta$, and~\eqref{e:XR-prop-discrete} since
$(\partial_\theta^2+1)|\sin\theta|=2(\delta(\theta)+\delta(\theta-\pi))$ and
$$
\mathcal H\varphi'(0)=-{1\over\pi}\int_{\mathbb R}x_1^{-2}\psi(x_1,0)\,dx_1=-{1\over \pi}T(|\mathbf x|^{-2}\psi)(0,0).
$$
\end{rema}
\begin{proof}
We express both sides of~\eqref{e:XR-Hilbert} in terms of the Fourier transform $\widehat\psi$.

\noindent 1. First of all, by the Fourier Inversion Formula we have
$$
\widehat\varphi(\xi_1)={1\over 2\pi}\int_{\mathbb R}\widehat\psi(\xi_1,\xi_2)\,d\xi_2.
$$
The Hilbert transform is a Fourier multiplier, more precisely
$$
\widehat{\mathcal H\varphi}(\xi_1)=-i(\sgn\xi_1) \widehat \varphi(\xi_1).
$$
Using the Fourier Inversion Formula again, we get a formula for the left-hand side of~\eqref{e:XR-Hilbert}:
\begin{equation}
  \label{e:XR-Hilbert-int-1}
\begin{aligned}
\mathcal H\varphi'(0)&={1\over 2\pi}\int_{\mathbb R}\widehat{\mathcal H\varphi'}(\xi_1)\,d\xi_1
={1\over 2\pi}\int_{\mathbb R}|\xi_1|\,\widehat\varphi(\xi_1)\,d\xi_1\\
&={1\over 4\pi^2}\int_{\mathbb R^2}|\xi_1|\,\widehat\psi(\xi_1,\xi_2)\,d\xi_1d\xi_2.
\end{aligned}
\end{equation}

\noindent 2. For a fixed $\theta\in\mathbb S^1$, the Fourier transform of $T(\Delta\psi)$ in the $s$ variable is given by
$$
\begin{aligned}
\int_{\mathbb R}e^{-i\eta s}T(\Delta\psi)(s,\theta)\,ds
&=\int_{\mathbb R^2}e^{-i\eta s}\Delta\psi(t\cos\theta-s\sin\theta,t\sin\theta+s\cos\theta)\,dtds\\
&=\int_{\mathbb R^2}e^{-i\eta(x_2\cos\theta-x_1\sin\theta)}\Delta\psi(x_1,x_2)\,dx_1dx_2\\
&=-\eta^2\widehat\psi(-\eta\sin\theta,\eta\cos\theta).
\end{aligned}
$$
Here in the second line we make the linear change of variables $(x_1,x_2)=\mathcal R_\theta(t,s)$ where $\mathcal R_\theta$ is the rotation
defined in~\eqref{e:R-gamma-def}.

By the Fourier Inversion Formula, we get the following expression of the X-ray transform $T(\Delta\psi)(0,\theta)$
in terms of the Fourier transform of~$\psi$:
\begin{equation}
  \label{e:XR-Hilbert-int-2}
T(\Delta\psi)(0,\theta)=-{1\over 2\pi}\int_{\mathbb R}\eta^2\,\widehat\psi(-\eta\sin\theta,\eta\cos\theta)\,d\eta.
\end{equation}

\noindent 3. We now write the integral~\eqref{e:XR-Hilbert-int-1} in polar coordinates $(\xi_1,\xi_2)=(-\eta\sin\theta,\eta\cos\theta)$:
$$
\begin{aligned}
\mathcal H\varphi'(0)&={1\over 4\pi^2}\int_0^\infty \eta^2\int_{\mathbb S^1}|\sin\theta|\,\widehat\psi(-\eta\sin\theta,\eta\cos\theta)\,d\theta d\eta\\
&={1\over 8\pi^2}\int_{\mathbb R\times\mathbb S^1}\eta^2|\sin\theta|\,\widehat\psi(-\eta\sin\theta,\eta\cos\theta)\,d\theta d\eta\\
&=-{1\over 4\pi}\int_{\mathbb S^1}|\sin\theta|\,T(\Delta\psi)(0,\theta)\,d\theta
\end{aligned}
$$
which gives~\eqref{e:XR-Hilbert}. Here in the second line we use that the inner integral $d\theta$ is an even function of $\eta\in\mathbb R$
and in the last line we use~\eqref{e:XR-Hilbert-int-2}.
\end{proof}

\subsection{Plurisubharmonic functions}

We review the concept of plurisubharmonicity on~$\mathbb C^2$. We first give the classical definition:
\begin{defi}
\label{d:psh}
Let $u:\mathbb C^2\to [-\infty,\infty)$ be an upper semicontinuous function not identically
equal to~$-\infty$. We say $u$ is plurisubharmonic
if for each $\mathbf z,\mathbf v\in\mathbb C^2$ we have the sub-mean value property
\begin{equation}
  \label{e:psh}
u(\mathbf z)\leq {1\over 2\pi}\int_{\mathbb S^1} u(\mathbf z+e^{i\theta}\mathbf v)\,d\theta.
\end{equation}
\end{defi}
In this paper we will instead use the alternative definition in terms of partial derivatives
defined using the operators (where we write points in $\mathbb C^2$ as
$\mathbf z=(x_1+iy_1,x_2+iy_2)$)
$$
\partial_{z_j}=\tfrac12(\partial_{x_j}-i\partial_{y_j}),\quad
\partial_{\bar z_j}=\tfrac12(\partial_{x_j}+i\partial_{y_j}).
$$
For $u\in C^2(\mathbb C^2;\mathbb R)$ we define the Hermitian matrix
\begin{equation}
  \label{e:d-dbar}
\partial_{\bar{\mathbf z}\mathbf z}u:=\begin{pmatrix}
\partial_{\bar z_1}\partial_{z_1}u & \partial_{\bar z_1}\partial_{z_2}u \\
\partial_{\bar z_2}\partial_{z_1}u & \partial_{\bar z_2}\partial_{z_2}u
\end{pmatrix}.
\end{equation}
In terms of real derivatives $\partial_{x_j},\partial_{y_j}$ the entries of this matrix take the form
\begin{equation}
  \label{e:d-bar-real}
\begin{aligned}
\partial_{\bar z_1}\partial_{z_1}u&=\tfrac 14
(\partial_{x_1}^2+\partial_{y_1}^2)u,\\
\partial_{\bar z_1}\partial_{z_2}u&= \tfrac14\big((\partial_{x_1}\partial_{x_2}+\partial_{y_1}\partial_{y_2})u+i(\partial_{y_1}\partial_{x_2}-\partial_{x_1}\partial_{y_2})u\big), \\
\partial_{\bar z_2}\partial_{z_1}u&=
\tfrac14\big((\partial_{x_1}\partial_{x_2}+\partial_{y_1}\partial_{y_2})u+i(\partial_{x_1}\partial_{y_2}-\partial_{y_1}\partial_{x_2})u\big),\\
\partial_{\bar z_2}\partial_{z_2}u&=\tfrac14(\partial_{x_2}^2+\partial_{y_2}^2)u.
\end{aligned}
\end{equation}
Then the $C^2$ function~$u$ is plurisubharmonic if and only if the matrix $\partial_{\bar{\mathbf z}\mathbf z}u(\mathbf z)$ is nonnegative for each $\mathbf z\in\mathbb C^2$, namely
$$
\langle\partial_{\bar{\mathbf z}\mathbf z}u(\mathbf z)\mathbf v,\mathbf v\rangle_{\mathbb C^2}\geq 0\quad\text{for all }
\mathbf z,\mathbf v\in\mathbb C^2,
$$
where $\langle \mathbf z,\mathbf w\rangle_{\mathbb C^2}:=z_1\overline{w_1}+z_2\overline{w_2}$ is the standard Hermitian inner product on~$\mathbb C^2$.

The functions we use later will not be in~$C^2$. However, for general functions there is still a second derivative
criterion for plurisubharmonicity if one uses distributional derivatives~\cite[Theorem~4.1.11]{Hormander1}.
In the special case of continuous functions it takes the following form:
\begin{prop}
  \label{l:psh-crit}
Let $u\in C^0(\mathbb C^2;\mathbb R)$. Then $u$ is plurisubharmonic if and only if for each $\mathbf v\in\mathbb C^2$
we have
\begin{equation}
  \label{e:psh-crit}
\langle\partial_{\bar{\mathbf z}\mathbf z}u(\mathbf z)\mathbf v,\mathbf v\rangle_{\mathbb C^2}\geq 0
\end{equation}
where the left-hand side of~\eqref{e:psh-crit} is considered as a distribution in~$\mathcal D'(\mathbb C^2)$.
\end{prop}
Here for a distribution $v\in\mathcal D'(\mathbb C^2)$ we say that $v\geq 0$ if $(v,\varphi)\geq 0$
for all test functions $\varphi\in\CIc(\mathbb C^2)$ such that $\varphi\geq 0$, where the distributional pairing $(v,\varphi)=v(\varphi)$
is defined in the usual way, in particular if $v\in L^1_{\loc}(\mathbb C^2)$ then
$$
(v,\varphi):=\int_{\mathbb C^2}v(\mathbf z)\varphi(\mathbf z)\,dx_1dx_2dy_1dy_2.
$$


\subsection{Kohn--Nirenberg norms}

Here we prove some basic estimates featuring the Kohn--Nirenberg
norm $\|\bullet\|_{S^1_3}$ defined in~\eqref{e:KN-def}; these are used in~\S\ref{s:weight-mod} below.
We denote by $K$ a global constant whose precise value may change from place to place.

Let $k\geq 1$ and define the rescaling map
$$
\Lambda_k:\mathbb R^2\to\mathbb R^2,\quad
\Lambda_k(\mathbf x):=2^k\mathbf x.
$$
Assume that
$$
\psi\in\CIc(\mathbb R^2),\quad
\supp\psi\subset \{2^{k-1}\leq |\mathbf x|\leq 2^{k+1}\}.
$$
Then the pullback $\Lambda_k^*\psi:=\psi\circ\Lambda_k$ is supported in $\{\tfrac 12\leq |\mathbf x|\leq 2\}$, and we have
\begin{equation}
  \label{e:KN-dilated}
K^{-1} 2^k\|\psi\|_{S^1_3}\leq \|\Lambda_k^*\psi\|_{C^3(\mathbb R^2)}\leq K2^k\|\psi\|_{S^1_3}.
\end{equation}
Since $T(|\mathbf x|^{-2}\psi)(0,\theta)=2^{-k}T(|\mathbf x|^{-2}\Lambda_k^*\psi)(0,\theta)$,
we get the following bound on the X-ray transform of $|\mathbf x|^{-2}\psi$
on lines through the origin:
\begin{equation}
  \label{e:XR-KN-OK}
\|T(|\mathbf x|^{-2}\psi)(0,\theta)\|_{C^3(\mathbb S^1)}\leq K \|\psi\|_{S^1_3}.
\end{equation}
We also have the following bound: if
$\chi\in\CIc((\tfrac 12,2))$, $F\in C^\infty(\mathbb S^1)$, and we use polar coordinates
$\mathbf x=(r\cos\theta,r\sin\theta)$ with $r\geq 0$, $\theta\in\mathbb S^1$,
then
\begin{equation}
  \label{e:XR-KN-OK2}
\|\chi(2^{-k}r)F(\theta)\|_{S^1_3}\leq K2^{-k}\|\chi\|_{C^3(\mathbb R)}\|F\|_{C^3(\mathbb S^1)}.
\end{equation}

\subsection{Line porous sets}

We finally prove several statements about line porous sets.
We start with a result about porous subsets of $\mathbb R$;
note that porosity on lines is equivalent to porosity on balls in dimension~$d=1$.
\begin{lemm}
\label{l:porosity-count-1D}
Assume that $Y\subset\mathbb R$ is $\nu$-porous on scales $\rho_0$ to $\rho_1$
and $I\subset\mathbb R$ is an interval of length $|I|\leq \rho_1$.
Then the Lebesgue measure of~$Y\cap I$ satisfies the bound
\begin{equation}
  \label{e:porosity-count-1D}
|Y\cap I|\leq |I|^\delta\rho_0^{1-\delta}
\end{equation}
for some constant $\delta\in (0,1)$ depending only on~$\nu$.
\end{lemm}
\begin{proof}
We show that~\eqref{e:porosity-count-1D} holds with
\begin{equation}
  \label{e:delta-def}
\delta:={\log 2\over\log 2-\log(1-\nu)}.
\end{equation}
We use induction on scales.
First of all, if $|I|\leq \rho_0$ then~\eqref{e:porosity-count-1D} holds since
$$
|Y\cap I|\leq |I|\leq |I|^\delta\rho_0^{1-\delta}.
$$
Now it is enough to show that if $|I|=R\in[\rho_0,\rho_1]$, and~\eqref{e:porosity-count-1D}
holds for all intervals of length $\leq (1-\nu)R$, then it holds for $I$.

Since $Y$ is $\nu$-porous on scales $\rho_0$ to $\rho_1$, there exists
a subinterval $J\subset I$ such that $|J|=\nu|I|$ and $Y\cap J=\emptyset$.
We write $I$ as the union of three nonoverlapping intervals
$$
I=I_1\cup J\cup I_2,\quad
|I_1|+|I_2|=(1-\nu)|I|
$$
and estimate
$$
\begin{aligned}
|Y\cap I|&\leq |Y\cap I_1|+|Y\cap I_2|\\
&\leq \big(|I_1|^\delta+|I_2|^\delta\big)\rho_0^{1-\delta}\\
&\leq 2^{1-\delta}(1-\nu)^\delta|I|^\delta\rho_0^{1-\delta}\\
&\leq |I|^\delta\rho_0^{1-\delta}
\end{aligned}
$$
which gives~\eqref{e:porosity-count-1D} for the interval~$I$.
Here in the first inequality we use that $Y\cap J=\emptyset$,
in the second inequality we use~\eqref{e:porosity-count-1D} for the intervals $I_1,I_2$,
in the third inequality we use the elementary statement
$$
s^\delta+(1-s)^\delta\leq 2^{1-\delta}\quad\text{for all }s\in [0,1]
$$
with $s:=|I_1|/((1-\nu)|I|)$, and in the last inequality we use the definition of~$\delta$.
\end{proof}
The next lemma shows that small neighborhoods of line porous sets are line porous:
\begin{lemm}
\label{l:porous-nbhd}
Assume that $\mathbf Y\subset\mathbb R^d$ is $\nu$-porous on lines on scales~$\rho_0$ to~$\rho_1$
and $0\leq\rho<\nu\rho_0$. Then the $\rho$-neighborhood
$\mathbf Y+B(0,\rho)$ is $\nu'$-porous on lines on scales~$\rho_0$ to~$\rho_1$,
where $\nu':=\nu-\rho/\rho_0$.
\end{lemm}
\begin{proof}
Let $\tau\subset\mathbb R^d$ be a line segment of length $R\in [\rho_0,\rho_1]$.
From line porosity of $\mathbf Y$ we see that there exists $\mathbf x\in\tau$
such that $\mathbf Y\cap B(\mathbf x,\nu R)=\emptyset$. We have
$\nu' R+\rho \leq \nu R$, thus
$(\mathbf Y+B(0,\rho))\cap B(\mathbf x,\nu' R)=\emptyset$,
showing line porosity of $\mathbf Y+B(0,\rho)$. 
\end{proof}

Combining the above two lemmas, we get the following counting statement, used in~\S\ref{s:weight-constr} below:
\begin{lemm}
\label{l:porosity-count}
Assume that $\mathbf Y\subset\mathbb R^2$ is $\nu$-porous on lines on scales
$\rho_0$ to $\rho_1$, and let $\rho_0\leq r_0\leq r_1\leq \rho_1$.
Let $\tau\subset\mathbb R^2$ be a line segment of length $r_1$.
Let $S\subset\mathbf Y$ be an $r_0$-separated set, that is $|\mathbf x-\mathbf x'|> r_0$
for all $\mathbf x,\mathbf x'\in S$, $\mathbf x\neq\mathbf x'$. Then
\begin{equation}
  \label{e:porosity-count}
\#\big(S\cap (\tau+B(0,2r_0))\big)\leq C\Big({r_1\over r_0}\Big)^{\delta},
\end{equation}
for some constants $C>0$, $\delta<1$ depending only on $\nu$. See Figure~\ref{f:porosity-count}.
\end{lemm}
\begin{figure}
\includegraphics{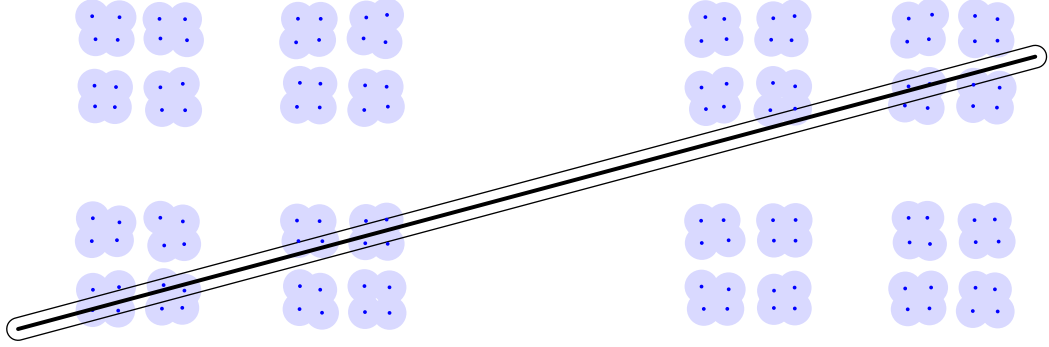}
\caption{An illustration of Lemma~\ref{l:porosity-count} for $S=\mathbf Y$. The dots
are the elements of the $r_0$-separated set $S$.
The thick line is the segment $\tau$, which has length~$r_1$.
The oval shape is the neighborhood $\tau+B(0,2r_0)$, and the blue shape is the neighborhood
$\mathbf Y+B(0,3r_0)$.}
\label{f:porosity-count}
\end{figure}
\begin{proof}
1. Consider the set
$$
\widetilde Y:= (\mathbf Y+B(0,3r_0))\cap\tau.
$$
By Lemma~\ref{l:porous-nbhd}, the set $\widetilde Y$ is $\nu\over 2$-porous
on lines on scales $\widetilde r_0$ to $\rho_1$, where
$$
\widetilde r_0:={6\over\nu}r_0.
$$
Taking a unit speed parametrization of the line segment $\tau$, we can consider
$\widetilde Y$ as a subset of $[0,r_1]\subset\mathbb R$.
This set is $\nu\over 2$-porous on scales $\widetilde r_0$ to~$\rho_1$.
Thus by Lemma~\ref{l:porosity-count-1D} with $I:=[0,r_1]$, we have, with
$\delta$ given by~\eqref{e:delta-def} with $\nu$ replaced by $\nu\over 2$, and $|\bullet|$ denoting the one-dimensional Lebesgue
measure on~$\tau$,
\begin{equation}
  \label{e:porosity-count-2d-1}
|\widetilde Y|\leq r_1^\delta {\widetilde r_0}^{\,1-\delta}.
\end{equation}
(Strictly speaking the above argument does not apply when $\widetilde r_0>\rho_1$,
but in the latter case we have $\widetilde r_0> r_1$
and~\eqref{e:porosity-count-2d-1} is immediate since $|\widetilde Y|\leq |\tau|=r_1$.)

\noindent 2. Denote $\widetilde S:=S\cap (\tau+B(0,2r_0))$.
For each $\mathbf x\in\widetilde S$, the set
$$
I_{\mathbf x}:=B(\mathbf x,3r_0)\cap\tau
$$
is a line segment of length $\geq r_0$ inside~$\tau$. Since
$\widetilde S\subset S\subset \mathbf Y$, we have $I_{\mathbf x}\subset \widetilde Y$.
Moreover, for any given $\mathbf y\in\tau$, if $\mathbf y\in I_{\mathbf x}$ then $\mathbf x\in B(\mathbf y,3r_0)$. 
Since $\widetilde S$ is an $r_0$-separated set,
the ball $B(\mathbf y,3r_0)$ contains at most $100$ points
in the set $\widetilde S$ (since the $r_0\over 2$-balls centered at these points
are nonintersecting and contained in the ball $B(\mathbf y,4r_0)$). It follows that each point $\mathbf y\in\tau$ is contained
in at most 100 intervals $I_{\mathbf x}$. We now estimate
$$
r_0\#(\widetilde S)\leq \sum_{\mathbf x\in \widetilde S}|I_{\mathbf x}|
\leq 100\bigg|\bigcup_{\mathbf x\in\widetilde S}I_{\mathbf x}\bigg|\leq 100|\widetilde Y|.
$$
Using~\eqref{e:porosity-count-2d-1}, we get 
$$
\#(\widetilde S)\leq {600\over\nu}\Big({r_1\over r_0}\Big)^\delta,
$$
giving~\eqref{e:porosity-count}.
\end{proof}

\section{Construction of a weight adapted to a line porous set}
\label{s:weight-constr}

In this section we prove Theorem~\ref{t:existence-weight} following~\cite[\S6.1]{Cohen-FUP-v1}.

Let $\nu$ be the line porosity constant of the set~$\mathbf Y$
and let $\delta\in (0,1)$ be the constant in Lemma~\ref{l:porosity-count}.
Fix constants $\eta,\alpha$ such that
\begin{equation}
  \label{e:alpha-gamma-fix}
0<\eta<\tfrac 13,\quad
3\eta\leq \alpha<1,\quad
\alpha>1-\eta(1-\delta).
\end{equation}
To construct the weight $\omega_k$, take $k\geq 1$, put
$$
r_k:={2^k\over k^\eta}
$$
and let
$$
S_k\ \subset\ \mathbf Y\cap \{2^{k-1}\leq |\mathbf x|\leq 2^{k+1}\}
$$
be a maximal $r_k$-separated subset.

Fix cutoff functions
$$
\begin{gathered}
\chi,\psi\in\CIc(\mathbb R^2;[0,1]),\quad
\supp\chi\subset B(0,2),\quad
\chi=1\text{ on }B(0,1);\\
\supp\psi\subset \{\tfrac 12<|\mathbf x|<2\},\quad
\psi=1\text{ on }\{\tfrac 1{\sqrt 2}\leq |\mathbf x|\leq \sqrt 2\}
\end{gathered}
$$
and put
\begin{equation}
  \label{e:omega-k-def}
\omega_k(\mathbf x):=-10\psi(2^{-k}\mathbf x)\sum_{\mathbf y\in S_k}{2^k\over k^\alpha}\,\chi\Big(
{\mathbf x-\mathbf y\over r_k}\Big).
\end{equation}
We show that $\omega_k$ satisfy all the conclusions of Theorem~\ref{t:existence-weight}.
Below $K$ denotes a global constant whose value may change from place to place.
\begin{itemize}
\item \eqref{e:omega-k-ass-1} holds since if $2^{k-1}>h^{-1}$, then
$S_k=\emptyset$ and thus $\omega_k=0$.
\item \eqref{e:omega-k-ass-2} follows immediately from the definition of~$\omega_k$.
\item To show the regularity property~\eqref{e:omega-k-ass-3},
we estimate for $|\gamma|\leq 3$ and all $\mathbf x\in\mathbb R^2$
$$
|\partial^\gamma_{\mathbf x}\omega_k(\mathbf x)|
\leq K {2^k\over k^\alpha}r_k^{-|\gamma|}\#(S_k\cap B(\mathbf x,2r_k))
\leq K2^{(1-|\gamma|)k}k^{|\gamma|\eta-\alpha}
$$
where in the last inequality we use that $S_k$ is an $r_k$-separated set
and thus the ball $B(\mathbf x,2r_k)$ contains no more than 100 points in $S_k$.
Since $|\gamma|\eta-\alpha\leq 3\eta-\alpha \leq 0$, we see
that $\|\omega_k\|_{S^1_3}\leq K$, giving~\eqref{e:omega-k-ass-3}.
\item To show the growth property~\eqref{e:omega-k-ass-4},
assume for now that $k$ satisfies $2^{k+2}\leq h^{-1}$.
For
$\theta\in \mathbb S^1$ consider the line segment
$$
\tau_{k,\theta}:=\{(t\cos\theta,t\sin\theta)\mid -2^{k+1}\leq t\leq 2^{k+1}\}\ \subset\ \mathbb R^2.
$$
For each $\mathbf y\in S_k$, if the support of $\chi\big({\mathbf x-\mathbf y\over r_k}\big)$ intersects $\tau_{k,\theta}$, then $\mathbf y$ lies in the set
$$
\widetilde S_{k,\theta}:=S_k\cap (\tau_{k,\theta}+B(0,2r_k)).
$$
By Lemma~\ref{l:porosity-count} with $\rho_0:=1$, $\rho_1:=h^{-1}$,
$r_0:=r_k\geq 1$, $r_1:=2^{k+2}$,
and the fact that $\mathbf Y$ is porous on lines, we have
$$
\#(\widetilde S_{k,\theta})\leq Ck^{\delta\eta}
$$
where the constant $C$ depends only on~$\nu$.
We now estimate (where $ds$ stands for arc length element)
$$
\begin{aligned}
-T(|\mathbf x|^{-2}\omega_k)(0,\theta)&\leq 
-2^{2-2k}\int_{\tau_{k,\theta}}\omega_k(\mathbf x)\,ds(\mathbf x)\\
&\leq K\cdot 2^{-2k}\cdot\#(\widetilde S_{k,\theta})\cdot {2^k\over k^\alpha}\cdot  r_k\\
&\leq KCk^{(\delta-1)\eta-\alpha}.
\end{aligned}
$$
This implies that, with $q_k:=-\inf_{\theta\in\mathbb S^1}T(|\mathbf x|^{-2}\omega_k)(0,\theta)$, we have
\begin{equation}
  \label{e:q-k-esti}
q_k\leq KCk^{(\delta-1)\eta-\alpha}\quad\text{when }
2^{k+2}\leq h^{-1}.
\end{equation}
By the proof of~\eqref{e:omega-k-ass-1} above, we know that
$q_k=0$ when $2^{k-1}>h^{-1}$. Thus~\eqref{e:q-k-esti}
is satisfied for all but 3 values of~$k$. Since each $q_k$
is bounded by~\eqref{e:q-k-apriori} and $(\delta-1)\eta-\alpha<-1$, we get
$$
\sum_{k=1}^\infty q_k\leq KC\sum_{k=1}^\infty k^{(\delta-1)\eta-\alpha}\leq KC,
$$
which gives~\eqref{e:omega-k-ass-4}.
\item It remains to show the property~\eqref{e:omega-large-Y}.
Take arbitrary $\mathbf x\in \mathbf Y\setminus B(0,10)$.
Choose $k\geq 1$ such that
\begin{equation}
\label{e:x-k-rest}
2^{k-\frac12}\leq |\mathbf x|\leq 2^{k+\frac12}.
\end{equation}
Then $\psi(2^{-k}\mathbf x)=1$. Since $S_k$ is a maximal $r_k$-separated
subset of $\mathbf Y\cap \{2^{k-1}\leq |\mathbf x|\leq 2^{k+1}\}$,
there exists $\mathbf y\in S_k$ such that $|\mathbf x-\mathbf y|\leq r_k$.
We have $\chi\big({\mathbf x-\mathbf y\over r_k}\big)=1$, thus from the definition~\eqref{e:omega-k-def} of~$\omega_k$ we see that
$$
\omega(\mathbf x)\leq \omega_k(\mathbf x)\leq -10{2^k\over k^\alpha}\leq -{|\mathbf x|\over(\log|\mathbf x|)^{\alpha}}
$$
where in the last inequality we use~\eqref{e:x-k-rest}.
This finishes the proof of~\eqref{e:omega-large-Y}.
\end{itemize}

\section{Plurisubharmonicity of the extended weight}
\label{s:psh-weight}

This section (following~\cite[\S3]{Cohen-FUP-v1})
gives a criterion for plurisubharmonicity of functions of the form~\eqref{e:u-intro-form}
in terms of the X-ray transform. This is the first ingredient in the proof of
Theorem~\ref{t:psh-constr}.
\begin{prop}
\label{l:psh-criterion}
Assume that $\omega\in\CIc(\mathbb R^2;\mathbb R)$ and $\sigma\in\mathbb R$. Define the function
$u\in C^0(\mathbb C^2;\mathbb R)$ by
\begin{equation}
  \label{e:psh-u-def}
u(\mathbf x+i\mathbf y)=E\omega(\mathbf x+i\mathbf y)+\sigma |\mathbf y|,\quad
\mathbf x,\mathbf y\in\mathbb R^2.
\end{equation}
Then $u$ is plurisubharmonic on~$\mathbb C^2$ if and only if
\begin{equation}
  \label{e:psh-criterion}
T(\Delta\omega)(s,\theta)+\pi\sigma\geq 0\quad\text{for all }
s\in\mathbb R,\ \theta\in\mathbb S^1.
\end{equation}
\end{prop}
\begin{rema}
\label{r:Hilbert}
In~\cite[(3.5)]{Cohen-FUP-v1} there is an extra condition needed for plurisubharmonicity,
in terms of the derivative of the Hilbert transform of restrictions of~$\omega$ to lines in~$\mathbb R^2$. More precisely, one requires that
for any $\mathbf z_0,\mathbf v_0\in\mathbb R^2$ with $|\mathbf v_0|=1$, if we define
$\omega_0(t):=\omega(\mathbf z_0+t\mathbf v_0)$, then
\begin{equation}
  \label{e:cohen-extra}
\mathcal H\omega_0'(0)\leq \sigma.
\end{equation}
This condition corresponds to subharmonicity of the restriction of~$u$ to the complex line
$\{\mathbf z_0+w\mathbf v_0\mid w\in\mathbb C\}$.
Our proof uses the distributional criterion for
plurisubharmonicity, Proposition~\ref{l:psh-crit}, and does not explicitly need the condition~\eqref{e:cohen-extra}.

However, the extra condition~\eqref{e:cohen-extra} can be derived from~\eqref{e:psh-criterion} as follows. Since \eqref{e:psh-criterion} is invariant under composition of $\omega$ with rotations and translations, we may assume that $\mathbf z_0=(0,0)$ and $\mathbf v_0=(1,0)$,
that is $\omega_0(t)=\omega(t,0)$. By Lemma~\ref{l:XR-Hilbert} we have
$$
\mathcal H\omega_0'(0)=-{1\over 4\pi}\int_{\mathbb S^1}|\sin\theta|\, T(\Delta\omega)(0,\theta)\,d\theta.
$$
Assume that $\omega$ satisfies~\eqref{e:psh-criterion}. Since $\int_{\mathbb S^1}|\sin\theta|\,d\theta=4$, we get the condition~\eqref{e:cohen-extra}.
\end{rema}

\subsection{An auxiliary distribution}

For each $t\in\mathbb R$, define the distribution
$$
Z_t(x_1+iy_1,x_2+iy_2):=\delta(x_1-ty_1)\delta(x_2-ty_2)\ \in\ \mathcal D'(\mathbb C^2).
$$
This family of distributions is integrable in~$t$, more precisely we have the following
\begin{lemm}
\label{l:Z-t-bdd}
If $\varphi\in\CIc(\mathbb C^2)$ and $\supp\varphi\subset B(0,R)$ for some $R\geq 0$, then
\begin{equation}
  \label{e:Z-t-bdd}
|(Z_t,\varphi)|\leq \pi R^2\|\varphi\|_{C^0}\,\langle t\rangle^{-2}.
\end{equation}
\end{lemm}
\begin{proof}
We have
\begin{equation}
\label{e:Z-t-pair}
(Z_t,\varphi)=\int_{\mathbb R^2}\varphi(ty_1+iy_1,ty_2+iy_2)\,dy_1dy_2.
\end{equation}
The expression under the integral is supported in the ball
$B(0,R/\langle t\rangle)$ and is bounded in absolute value by~$\|\varphi\|_{C^0}$,
which gives~\eqref{e:Z-t-bdd}.
\end{proof}
We also have the following identity featuring the matrix $\partial_{\bar{\mathbf z}\mathbf z}Z_t$: 
\begin{equation}
\label{e:Z-t-ders}
\partial_{\bar {\mathbf z}\mathbf z}Z_t={1+t^2\over 4}\,
\partial_{\mathbf x}^2Z_t,\qquad
\partial_{\mathbf x}^2:=\begin{pmatrix}
\partial_{x_1}^2  & \partial_{x_1x_2}  \\
\partial_{x_1x_2}  & \partial_{x_2}^2 
\end{pmatrix}.
\end{equation}

\subsection{The extension operator as a convolution}

We now write $E\omega$ as a distributional convolution. For
$\psi\in\CIc(\mathbb R^2)$, define the distribution
$$
(\psi\otimes\delta)(x_1+iy_1,x_2+iy_2):=\psi(x_1,x_2)\delta(y_1)\delta(y_2)\ \in\ \mathcal E'(\mathbb C^2).
$$
Then for each $t\in\mathbb R$ we have
\begin{equation}
  \label{e:Z-t-conv}
(Z_t*(\psi\otimes\delta))(\mathbf x+i\mathbf y)=\psi(\mathbf x-t\mathbf y).
\end{equation}
By Lemma~\ref{l:Z-t-bdd}, the pairing of~\eqref{e:Z-t-conv}
with any test function $\varphi\in\CIc(\mathbb C^2)$ is $\mathcal O(\langle t\rangle^{-2})$.
In fact, we have the following bound on this distribution in~$L^1_{\loc}(\mathbb C^2)$:
\begin{lemm}
  \label{l:Z-t-conv-L1}
Let $R>0$ and assume that $\supp\psi\subset B(0,R)$. Then we have
\begin{equation}
  \label{e:Z-t-conv-L1}
\|Z_t*(\psi\otimes\delta)\|_{L^1(B(0,R))}\leq 50R^4\|\psi\|_{C^0}\, \langle t\rangle^{-2}.
\end{equation}
\end{lemm}
\begin{proof}
By~\eqref{e:Z-t-conv}, the left-hand side of~\eqref{e:Z-t-conv-L1} is equal to
\begin{equation}
  \label{e:Z-t-conv-L1-int}
\int_{B(0,R)}|\psi(\mathbf x-t\mathbf y)|\,d\mathbf xd\mathbf y.
\end{equation}
If $(\mathbf x,\mathbf y)\in B(0,R)$ and $\psi(\mathbf x-t\mathbf y)\neq 0$,
then
$$
|\mathbf x|,|\mathbf y|,|\mathbf x-t\mathbf y|\leq R,
$$
which in turn implies that $|t\mathbf y|\leq 2R$ and thus
$|\mathbf y|\leq \sqrt 5 R\langle t\rangle^{-1}$. Thus we can bound the integral
in~\eqref{e:Z-t-conv-L1-int} by the integral over $\{|\mathbf x|\leq R,\ |\mathbf y|\leq \sqrt 5R\langle t\rangle^{-1}\}$, which is a set of volume $5\pi^2 R^4\langle t\rangle^{-2}\leq 50 R^4\langle t\rangle^{-2}$. This gives~\eqref{e:Z-t-conv-L1}.
\end{proof}
By Lemma~\ref{l:Z-t-conv-L1}, the function $Z_t*(\psi\otimes\delta)$ is integrable in~$t$
with values in $L^1_{\loc}(\mathbb C^2)$, which by Fubini's Theorem is equivalent
to $Z_t*(\psi\otimes\delta)$ being integrable on $\mathbb R_t\times B(0,R)_{(\mathbf x,\mathbf y)}$
for each~$R$. The next lemma computes the integral of $Z_t*(\psi\otimes\delta)$
in~$t$ in terms of the X-ray transform of~$\psi$:
\begin{lemm}
  \label{l:Z-t-integrated}
For any $\psi\in\CIc(\mathbb R^2)$, we have
\begin{equation}
\label{e:Z-t-integrated}
\int_{\mathbb R}(Z_t*(\psi\otimes\delta))(\mathbf x+i\mathbf y)\,dt={1\over r}\,T\psi(x_2\cos\theta-x_1\sin\theta,\theta)
\end{equation}
where we write $\mathbf x=(x_1,x_2)$, $\mathbf y=(r\cos\theta,r\sin\theta)$ with
$r>0$, $\theta\in\mathbb S^1$, and
the integral is understood in $L^1_{\loc}(\mathbb C^2)$.
\end{lemm}
\begin{proof}
It suffices to check that the formula~\eqref{e:Z-t-integrated} holds
pointwise for each $\mathbf x\in\mathbb R^2$, $\mathbf y\in\mathbb R^2\setminus \{0\}$.
Using~\eqref{e:Z-t-conv} we compute
$$
\begin{gathered}
\int_{\mathbb R}(Z_t*(\psi\otimes\delta))(\mathbf x+i\mathbf y)\,dt
=\int_{\mathbb R}\psi(x_1-rt\cos\theta,x_2-rt\sin\theta)\,dt\\
={1\over r}\int_{\mathbb R}\psi(\tilde t\cos\theta-(x_2\cos\theta-x_1\sin\theta)\sin\theta,
\tilde t\sin\theta+(x_2\cos\theta-x_1\sin\theta)\cos\theta)\,d\tilde t\\
={1\over r}\,T\psi(x_2\cos\theta-x_1\sin\theta,\theta)
\end{gathered}
$$
giving~\eqref{e:Z-t-integrated}.
Here in the second line we make the change of variables $\tilde t=-rt+x_1\cos\theta+x_2\sin\theta$.
\end{proof}
Using the above lemmas, we now compute
the matrix of distributional derivatives $\partial_{\bar{\mathbf z}\mathbf z}E\omega$:
\begin{lemm}
\label{l:E-ddbar}
Let $\omega\in\CIc(\mathbb R^2)$. Then the entries of the matrix $\partial_{\bar{\mathbf z}\mathbf z}E\omega$ lie in $L^1_{\loc}(\mathbb C^2)$ and we have
\begin{equation}
\label{e:E-ddbar}
\partial_{\bar {\mathbf z}\mathbf z}E\omega(\mathbf x+i\mathbf y)={1\over 4 \pi r}\,
T(\partial_{\mathbf x}^2\omega)(x_2\cos\theta-x_1\sin\theta,\theta)
\end{equation}
where we write $\mathbf x=(x_1,x_2)$, $\mathbf y=(r\cos\theta,r\sin\theta)$ with
$r>0$, $\theta\in\mathbb S^1$.
\end{lemm}
\begin{proof}
Comparing~\eqref{e:E-def} with~\eqref{e:Z-t-conv}, we express $E\omega$ as a distributional convolution: for each $\omega\in\CIc(\mathbb R^2)$
$$
E\omega={1\over\pi}\int_{\mathbb R}{Z_t*(\omega\otimes\delta)\over 1+t^2}\,dt.
$$
Using~\eqref{e:Z-t-ders} and the properties of convolution under differentiation we compute
$$
\begin{aligned}
\partial_{\bar{\mathbf z}\mathbf z}E\omega&=
{1\over\pi}\int_{\mathbb R}{(\partial_{\bar{\mathbf z}\mathbf z}Z_t)*(\omega\otimes\delta)\over 1+t^2}\,dt\\
&=
{1\over 4\pi}\int_{\mathbb R}(\partial_{\mathbf x}^2 Z_t)*(\omega\otimes\delta)
\,dt\\
&=
{1\over 4\pi}\int_{\mathbb R}{Z_t*(\partial_{\mathbf x}^2\omega\otimes\delta)}\,dt.
\end{aligned}
$$
Here by Lemma~\ref{l:Z-t-conv-L1} the integrals above converge in $L^1_{\loc}(\mathbb C^2)$.
It remains to use Lemma~\ref{l:Z-t-integrated} for the functions $\psi:=\partial_{x_j}\partial_{x_k}\omega$.
\end{proof}

\subsection{End of the proof}
\label{s:psh-crit-proof}

We are now ready to give
\begin{proof}[Proof of Proposition~\ref{l:psh-criterion}]
1. We first compute
\begin{equation}
  \label{e:y-abs-der}
\partial_{\bar{\mathbf z}\mathbf z}|\mathbf y|={1\over 4r}A_\theta,\quad
A_\theta:=\begin{pmatrix}\sin^2\theta & -\sin\theta\cos\theta \\ -\sin\theta\cos\theta & \cos^2\theta\end{pmatrix}.
\end{equation}
Here we write $\mathbf y=(r\cos\theta,r\sin\theta)$ with $r>0$ and $\theta\in\mathbb S^1$, differentiation is understood
in the sense of distributions on~$\mathbb C^2$, and~\eqref{e:y-abs-der}
has entries in $L^1_{\loc}(\mathbb C^2)$. To see~\eqref{e:y-abs-der}
we first note that both sides only depend on $\mathbf y=\Im\mathbf z$
and thus can be considered as distributions on $\mathbb R^2$.
Now~\eqref{e:y-abs-der} can be verified by direct computation on $\mathbb R^2\setminus \{0\}$,
where the function $|\mathbf y|$ is smooth, and extends 
to an identity in the sense of distributions on the entire $\mathbb R^2$ by homogeneity, see for example~\cite[Theorem~3.2.3]{Hormander1}.

\noindent 2. Recall that $u=E\omega+\sigma|\mathbf y|$. Combining Lemma~\ref{l:E-ddbar}
and~\eqref{e:y-abs-der} we see that
\begin{equation}
  \label{e:u-total-der}
4\pi\partial_{\bar{\mathbf z}\mathbf z}u(\mathbf x+i\mathbf y)={1\over r}\big(T(\partial_{\mathbf x}^2\omega)(x_2\cos\theta-x_1\sin\theta,\theta)+\pi\sigma A_\theta\big).
\end{equation}
Here derivatives are understood in the sense of distributions on~$\mathbb C^2$
and~\eqref{e:u-total-der} has entries in $L^1_{\loc}(\mathbb C^2)$.

By Proposition~\ref{l:psh-crit}, we see that $u$ is plurisubharmonic if and only if
\begin{equation}
  \label{e:u-psh-int-1}
\big\langle \big(T(\partial_{\mathbf x}^2\omega)(s,\theta)+\pi\sigma A_\theta\big)\mathbf v,\mathbf v\big\rangle_{\mathbb C^2}\geq 0\quad\text{for all }
s\in\mathbb R,\
\theta\in\mathbb S^1,\
\mathbf v\in\mathbb C^2.
\end{equation}
Since $\omega$ is real-valued, the matrix on the left-hand side has real entries. Therefore,
it suffices to verify~\eqref{e:u-psh-int-1} for all vectors of the form $\mathbf v_\gamma=(\cos\gamma,\sin\gamma)$ where $\gamma\in\mathbb S^1$. Using Lemma~\ref{l:XR-prop-1} we compute
$$
\begin{aligned}
\big\langle \big(T(\partial_{\mathbf x}^2\omega)(s,\theta)+\pi\sigma A_\theta\big)\mathbf v_\gamma,\mathbf v_\gamma\big\rangle_{\mathbb C^2}
&=T(V_\gamma^2\omega)(s,\theta)+\pi\sigma \sin^2(\theta-\gamma)\\
&=\big(T(\Delta\omega)(s,\theta)+\pi\sigma\big)\sin^2(\theta-\gamma).
\end{aligned}
$$
It follows that~\eqref{e:u-psh-int-1} is equivalent to
$$
T(\Delta\omega)(s,\theta)+\pi\sigma\geq 0\quad\text{for all }s\in\mathbb R,\
\theta\in\mathbb S^1
$$
which finishes the proof.
\end{proof}

\section{Proof of Theorem~\ref{t:psh-constr}}
\label{s:psh-constr}

In this section we finish the proof of Theorem~\ref{t:psh-constr} by constructing
a weight $\widetilde\omega\leq\omega$ which satisfies the X-ray transform condition~\eqref{e:psh-criterion}.
We follow~\cite[\S4]{Cohen-FUP-v1}.

\subsection{Weights with constant radial integrals}
\label{s:weight-cr}

We start with the special case when for each $k$, the integral of
$|\mathbf x|^{-2}\omega_k$ on lines through the origin is constant
(i.e. does not depend on the choice of the line). In
this case the $\partial_\theta^2$ term in Lemma~\ref{l:XR-prop-2} disappears
and we can take $\widetilde\omega:=\omega$ in Theorem~\ref{t:psh-constr}.
\begin{prop}
  \label{l:weight-level}
Assume that $\omega$ is given by~\eqref{e:omega-form} where $\omega_k$ satisfy the
conditions~\eqref{e:omega-k-ass-1}--\eqref{e:omega-k-ass-4} for some constants $C_{\reg},C_{\gr}$.
Assume moreover that for all $k\geq 1$ and $\theta\in\mathbb S^1$
\begin{equation}
  \label{e:weight-level}
T(|\mathbf x|^{-2}\omega_k)(0,\theta)=-q_k
\end{equation}
where $q_k$ depends on $k$ but does not depend on~$\theta$. Then for $\sigma\geq 0$
large enough depending only on~$C_{\reg},C_{\gr}$ the function $u$ defined by~\eqref{e:psh-u-def}
is plurisubharmonic on~$\mathbb C^2$.
\end{prop}
\begin{proof}
Throughout the proof the letter $K$ denotes a global constant whose
value may change from place to place.

1. By Proposition~\ref{l:psh-criterion}, it suffices to show that for
$\sigma$ large enough depending only on~$C_{\reg},C_{\gr}$ we have
\begin{equation}
  \label{e:psh-crit-again}
T(\Delta\omega)(s,\theta)+\pi\sigma\geq 0\quad\text{for all }s\in\mathbb R,\
\theta\in\mathbb S^1.
\end{equation}
We first consider the special case $s=0$. In this case by Lemma~\ref{l:XR-prop-2}
and~\eqref{e:weight-level} we compute
\begin{equation}
  \label{e:weight-level-int-1}
T(\Delta\omega_k)(0,\theta)=-q_k\quad\text{for all }k\geq 1,\ \theta\in\mathbb S^1.
\end{equation}
Together with~\eqref{e:omega-k-ass-4} this shows that~\eqref{e:psh-crit-again} holds for $s=0$ as long
as $\pi\sigma\geq C_{\gr}$.

To handle the case of general $s$, we will show the following estimate
which only needs the conditions~\eqref{e:omega-k-ass-1}--\eqref{e:omega-k-ass-3}:
\begin{equation}
  \label{e:weight-level-int-2}
\bigg|T(\Delta\omega)(s,\theta)-\sum_{k\geq 1\atop 2^{k+1}\geq |s|}T(\Delta\omega_k)(0,\theta)\bigg|\leq K C_{\reg}.
\end{equation}
Together with~\eqref{e:weight-level-int-1}, \eqref{e:omega-k-ass-4}, and the fact that $q_k\geq 0$
this shows that~\eqref{e:psh-crit-again} holds as long as $\pi\sigma\geq C_{\gr}+KC_{\reg}$.

2. It remains to show~\eqref{e:weight-level-int-2}. By the discrete symmetry~\eqref{e:XR-prop-discrete}
and equivariance under rotations~\eqref{e:XR-prop-rotate} we reduce to the case when $\theta=0$ and $s\geq 0$,
that is it suffices to show that
\begin{equation}
  \label{e:weight-level-int-3}
\bigg|T(\Delta\omega)(s,0)-\sum_{k\geq 1\atop 2^{k+1}\geq s}T(\Delta\omega_k)(0,0)\bigg|\leq KC_{\reg}.
\end{equation}
Fix $s\geq 0$. If $k$ is such that $2^{k+1}<s$, then by~\eqref{e:omega-k-ass-2} the
support of $\Delta\omega_k$ lies in the ball $B(0,s)$ and thus
$T(\Delta\omega_k)(s,0)=0$. Thus~\eqref{e:weight-level-int-3} becomes
\begin{equation}
  \label{e:weight-level-int-4}
\bigg|\sum_{k\geq 1\atop 2^{k+1}\geq s}\int_{\mathbb R}\big(\Delta\omega_k(x_1,s)-\Delta\omega_k(x_1,0)\big)\,dx_1\bigg|\leq K C_{\reg}.
\end{equation}
By the triangle inequality and the fundamental theorem of calculus in the $x_2$ variable,
the left-hand side of~\eqref{e:weight-level-int-4} is bounded by
$$
\sum_{k\geq 1\atop 2^{k+1}\geq s}\int_{\mathbb R\times [0,s]}\big|\partial_{x_2} \Delta\omega_k(x_1,x_2)\big|\,dx_1dx_2.
$$
By~\eqref{e:omega-k-ass-3} and~\eqref{e:omega-k-ass-2} we have
$\sup|\partial_{x_2}\Delta\omega_k|\leq K C_{\reg}2^{-{2k}}$.
On the other hand, by~\eqref{e:omega-k-ass-2} the intersection of~$\mathbb R\times [0,s]$ with
$\supp\partial_{x_2}\Delta\omega_k$ is contained in $[-2^{k+1},2^{k+1}]\times [0,s]$, which has
area $2^{k+2}s$. It follows that the left-hand side of~\eqref{e:weight-level-int-4} is bounded by
$$
KC_{\reg}\,s \sum_{k\geq 1\atop 2^{k+1}\geq s} 2^{-k}\leq KC_{\reg},
$$
finishing the proof of~\eqref{e:weight-level-int-2}.
\end{proof}

\subsection{Modifying a general weight}
\label{s:weight-mod}

We now give
\begin{proof}[Proof of Theorem~\ref{t:psh-constr}]
1. Assume that $\omega$ is given
by~\eqref{e:omega-form} where $\omega_k$ satisfy the
conditions~\eqref{e:omega-k-ass-1}--\eqref{e:omega-k-ass-4} for some constants $C_{\reg},C_{\gr}$.
We subtract a function from each $\omega_k$ to obtain a weight $\widetilde\omega_k$
satisfying the additional condition~\eqref{e:weight-level}.
To do this, fix a cutoff function
$$
\chi\in\CIc\big((-2,-\tfrac12)\cup (\tfrac 12,2);[0,\infty)\big),\quad
\chi(-t)=\chi(t),\quad
\int_{\mathbb R}t^{-2}\chi(t)\,dt=1.
$$
For each $k\geq 1$, define the function $\psi_k\in\CIc(\mathbb R^2\setminus \{0\})$ in polar coordinates:
for $t\in\mathbb R$ and $\theta\in\mathbb S^1$ put
$$
\psi_k(t\cos\theta,t\sin\theta):=2^k\chi(2^{-k}t)\big(T(|\mathbf x|^{-2}\omega_k)(0,\theta)+q_k\big).
$$
Now, define the modified weight
$$
\widetilde \omega:=\sum_{k\geq 1}\widetilde\omega_k\quad\text{where }
\widetilde\omega_k:=\omega_k-\psi_k.
$$

\noindent 2. We show several properties of $\widetilde\omega_k$. Recall from~\eqref{e:omega-k-ass-4} that
$$
q_k:=-\inf_{\theta\in\mathbb S^1}T(|\mathbf x|^{-2}\omega_k)(0,\theta).
$$
It follows
that $\psi_k\geq 0$ and thus $\widetilde\omega_k\leq\omega_k$. This shows that
$\widetilde\omega\leq\omega$.

It is easy to see that $\widetilde\omega_k$ satisfy the conditions~\eqref{e:omega-k-ass-1}--\eqref{e:omega-k-ass-2}.
Moreover, by~\eqref{e:XR-KN-OK} and~\eqref{e:XR-KN-OK2} we see that
$\|\widetilde\omega_k\|_{S^1_3}\leq K\|\omega_k\|_{S^1_3}$ for some global constant $K$.
Therefore
$\widetilde\omega_k$ satisfy the Kohn--Nirenberg regularity condition~\eqref{e:omega-k-ass-3} with constant $KC_{\reg}$ where $C_{\reg}$ is the constant in the regularity condition~\eqref{e:omega-k-ass-3} for~$\omega_k$.

Finally, we compute for each $\theta\in\mathbb S^1$
$$
T(|\mathbf x|^{-2}\widetilde\omega_k)(0,\theta)=T(|\mathbf x|^{-2}\omega_k)(0,\theta)-\int_{\mathbb R}
t^{-2}\,\psi_k(t\cos\theta,t\sin\theta)\,dt=-q_k.
$$
Therefore, $\widetilde\omega_k$ satisfy the additional condition~\eqref{e:weight-level}
and thus also the growth condition~\eqref{e:omega-k-ass-4} (with the same values of $q_k$
as for $\omega_k$).

\noindent 3. It remains to apply Proposition~\ref{l:weight-level} with $\omega_k$
replaced by $\widetilde\omega_k$ to see that
for $\sigma$ large enough depending only on~$C_{\reg},C_{\gr}$, the function
$\widetilde u$ defined by~\eqref{e:u-intro-form} is plurisubharmonic on~$\mathbb C^2$.
\end{proof}

\medskip\noindent\textbf{Acknowledgements.}
I am grateful to Alex Cohen for sharing the early versions of~\cite{Cohen-FUP-v1} with me
and many discussions.
I was supported by NSF CAREER grant DMS-1749858.

\bibliographystyle{alpha}
\bibliography{General,Dyatlov,QC}

\begin{thebibliography}{Dya19}

\bibitem[BD18]{fullgap}
Jean Bourgain and Semyon Dyatlov.
\newblock Spectral gaps without the pressure condition.
\newblock {\em Ann. of Math. (2)}, 187(3):825--867, 2018.

\bibitem[Coh23]{Cohen-FUP-v1}
Alex Cohen.
\newblock Fractal uncertainty in higher dimensions, 2023.
\newblock \arXiv{2305.05022v1}.

\bibitem[Dya19]{FUP-ICMP}
Semyon Dyatlov.
\newblock An introduction to fractal uncertainty principle.
\newblock {\em J. Math. Phys.}, 60(8):081505, 31, 2019.

\bibitem[H{\"o}r03]{Hormander1}
Lars H{\"o}rmander.
\newblock {\em The analysis of linear partial differential operators. {I}}.
\newblock Classics in Mathematics. Springer-Verlag, Berlin, 2003.
\newblock Distribution theory and Fourier analysis, Reprint of the second
  (1990) edition [Springer, Berlin; MR1065993 (91m:35001a)].

\bibitem[HS20]{Han-Schlag}
Rui Han and Wilhelm Schlag.
\newblock A higher-dimensional {B}ourgain-{D}yatlov fractal uncertainty
  principle.
\newblock {\em Anal. PDE}, 13(3):813--863, 2020.

\end{thebibliography}

\end{document}